\documentclass{jac}
\pdfoutput=1

\usepackage[utf8]{inputenc}
\usepackage{amssymb,amsmath,graphicx,tikz}
\newtheorem{opb}[thm]{Open Problem}

\begin{document}

\title{Combinatorial Substitutions and Sofic Tilings}

\author{Th. Fernique}{Thomas Fernique}
\author{N. Ollinger}{Nicolas Ollinger}
\address{Laboratoire d'informatique fondamentale de Marseille (LIF)
  \newline Aix-Marseille Universit\'e, CNRS,
  \newline 39 rue Joliot-Curie, 13\kern 0.2em 013 Marseille, France}
\email{{Thomas.Fernique,Nicolas.Ollinger}@lif.univ-mrs.fr}

\thanks{Partially supported by the ANR projects EMC (ANR-09-BLAN-0164) and SubTile}

\begin{abstract}
\noindent
A combinatorial substitution is a map over tilings which allows to define sets of tilings with a strong hierarchical structure.
In this paper, we show that such sets of tilings are sofic, that is, can be enforced by finitely many local constraints.
This extends some similar previous results (Mozes'90, Goodman-Strauss'98) in a much shorter presentation.
\end{abstract}

\maketitle

%%%%%%%%%%%%%%%%%%%%%%%
%%%%%%%%%%%%%%%%%%%%%%%
%%%%%%%%%%%%%%%%%%%%%%%
\section{Introduction}

Tiling some space with geometrical shapes, or tiles, consists into covering this space with copies of the tiles. When a set of tilings can be characterized by adding finitely many local constraints on tiles so that the set of tilings corresponds exactly to the set of tilings satisfying the constraints, such a set of tilings is called sofic. Soficity corresponds to the interesting idea that the validity of a tiling can be locally proved by decorating tiles with constraints. In this paper, we contribute to a general question: which sets of tilings are sofic?

\medskip
Substitutions provide a simple way to express how a set of tilings can be obtained by iteratively constructing bigger and bigger aggregates of tiles. The strong hierarchical structure of substitutions tilings permits to enforce global properties, for example aperiodicity. To prove that the set of tilings generated by a substitution is sofic, one has to encode the global hierarchical structure into local constraints on tiles. Such technique is at the root of classical papers on the undecidability of the Domino Problem \cite{RB,RR}. In the case of tilings on the square grid, Mozes~\cite{SM} proved in a seminal paper that the set of tilings generated by rectangular non-deterministic substitutions satisfying a particular property is sofic. Goodman-Strauss~\cite{GS} proved that such a construction method can be extended to a wide variety of geometrical substitutions. In this paper, based on ideas developed in \cite{NO}, we further extend the construction to a broader class of substitutions by replacing the geometrical conditions by combinatorial conditions while decreasing the length of the presentation.

\medskip
Let us sketch some definitions to state our main theorem. A sofic tiling is a valid tiling by a finite set of tiles with decorations. Combinatorial substitutions, introduced by Priebe-Frank in \cite{PF}, map a tiling by tiles onto a tiling by so-called macro-tiles (finite tilings, here assumed to be connected), so that the macro-tiles of the latter are arranged as the tiles of the former. The limit set of a substitution is the set of complete tilings that admits preimages of any depth by the substitution. A good substitution is a substitution that is both connecting (a combinatorial condition ensuring that there is enough room for all information to flow) and consistent (a geometrical condition ensuring that the substitution can be correctly iterated).
The main result obtained in this paper is:

\begin{theorem}\label{th:soficity}
The limit set of a good combinatorial substitution is sofic.
\end{theorem}

The paper is organized as follows.
Sections \ref{sec:sofic} and \ref{sec:substitution} formally define main notions, in particular sofic tilings and good combinatorial substitutions.
Section \ref{sec:simulation} then presents \emph{self-simulation}, which plays a central role in the constructive proof of Theorem~\ref{th:soficity}, which is given in section \ref{sec:proof}.
Last, section \ref{sec:counting} concludes the paper by discussing an important parameter of this proof.\medskip

A single example illustrates definitions and results throughout the whole paper: the ``Rauzy'' example.
It relies on the theory of \emph{generalized substitutions} introduced in \cite{AI}, a self-contained presentation of which is beyond the scope of this paper.
Let us just mention that \emph{Rauzy tilings} are digitizations of the planes of the Euclidean space with a specific given irrational normal vector.
More details, as well as the results we here implicitly rely on, can be found in \cite{BF,TF}.
We chose this example, not the simplest one, because it is not covered by results in \cite{SM,GS}.

%%%%%%%%%%%%%%%%%%%%%%%
%%%%%%%%%%%%%%%%%%%%%%%
%%%%%%%%%%%%%%%%%%%%%%%
\section{Sofic tilings}\label{sec:sofic}

Polytopes are assumed to be homeomorphic to closed balls of $\mathbb{R}^d$ and to have finitely many faces, with the $(d-1)$-dimensional faces being called \emph{facets}.

\medskip
A \emph{tile} $T$ is a polytope of the Euclidean space $\mathbb{R}^d$.
A \emph{tiling} $Q$ of a \emph{domain} $D\subset\mathbb{R}^d$ is a covering of $D$ by interior-disjoint tiles, with the additional condition that two tiles can intersect (if they do) only along entire faces.
A tiling is said to be \emph{finite} if its domain is bounded.

\medskip
Two tiles are said to be \emph{adjacent} if they intersect along at least one facet, and a tiling is said to be \emph{connected} if any two of its tiles can be connected by a sequence of adjacent tiles.

\medskip
A facet of a tiling is said to be \emph{external} if it is on the boundary of the domain, \emph{internal} otherwise.
We denote by $\partial Q$ the set of external facets of a tiling $Q$.
If this set is empty, then the tiling is said to be \emph{complete}: its domain is the whole $\mathbb{R}^d$.

\medskip
A \emph{decorated tile} $\mathcal{T}$ is a tile with a real map defined on its boundaries, called \emph{decoration}.
Two adjacent decorated tiles are said to \emph{match} if their decorations are equal in any point of their intersecting facets.
A \emph{decorated tiling} $\mathcal{Q}$ is a tiling by decorated tiles which pairwise match (when adjacent).
Decorations are thus local constraints on the way tiles can be arranged in a tiling.

\medskip
A \emph{tileset} $\tau$ is a set of decorated tiles.
It is said to be \emph{finite} if it contains only a finite number of tiles up to direct isometries.
A decorated tiling whose tiles belong to a tileset $\tau$ is called a $\tau$-tiling; the set of $\tau$-tilings is denoted by $\Lambda_\tau$.

\medskip
One says that a tiling \emph{can be seen as} a decorated tiling if both are equal up to decorations.
One denotes by $\pi$ the map which removes the decorations.
One easily checks that any tiling can be seen as a $\tau$-tiling if $\tau$ can be infinite.
The interesting case is the one of finite tilesets:

\begin{definition}\label{def:sofic}
A set of tilings is said to be \emph{sofic} if it can be seen as the set of $\tau$-tilings of some finite tileset $\tau$.
\end{definition}

\noindent For example, one can wonder whether the \emph{Rauzy tilings} are sofic (Fig.~\ref{fig:rauzy_tiling_a}).

\medskip
\begin{figure}[hbtp]
\centering
\includegraphics[width=\textwidth]{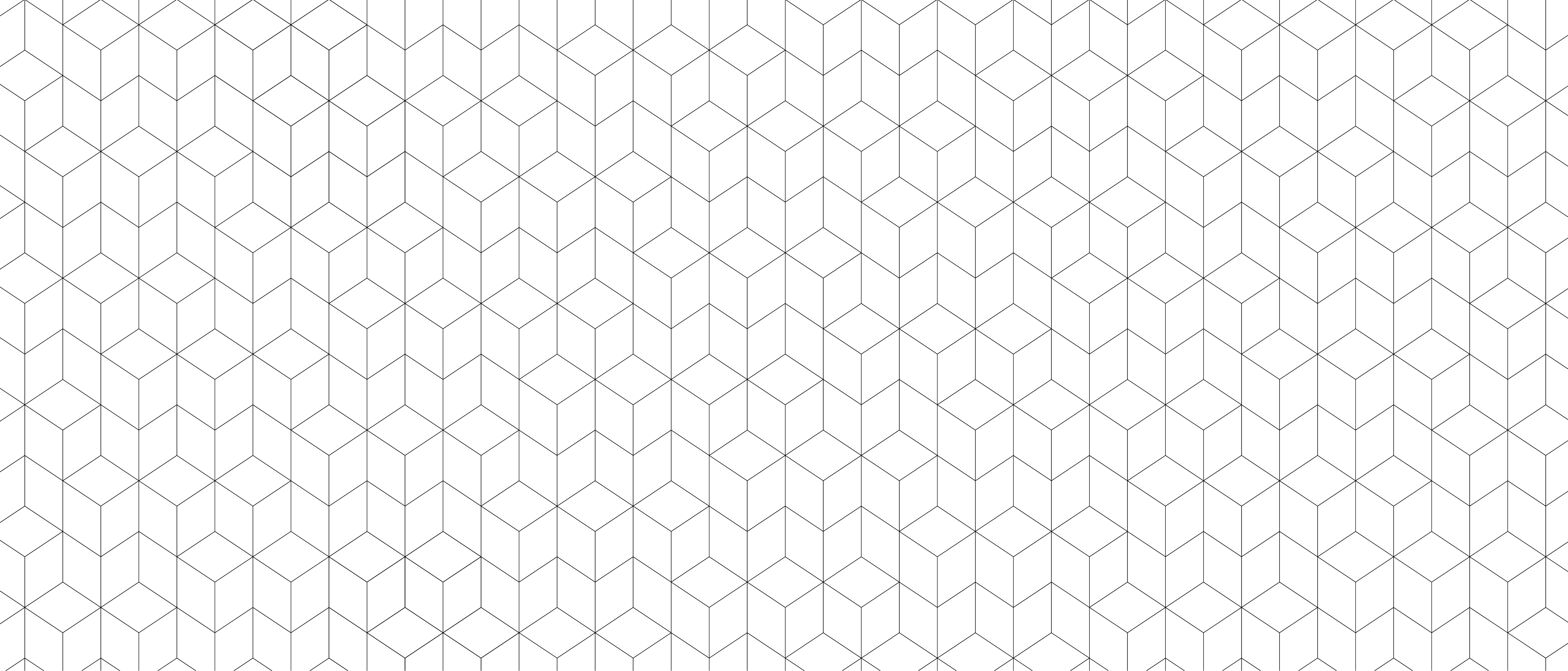}
\caption{A Rauzy tiling (partial view). Are Rauzy tilings sofic?}
\label{fig:rauzy_tiling_a}
\end{figure}

%%%%%%%%%%%%%%%%%%%%%%%
%%%%%%%%%%%%%%%%%%%%%%%
%%%%%%%%%%%%%%%%%%%%%%%
\section{Combinatorial substitutions}\label{sec:substitution}

\begin{definition}\label{def:substitution_rule}
A \emph{combinatorial substitution} is a finite set of \emph{rules} $(P,Q,\gamma)$, where $P$ is a tile, $Q$ is a finite connected tiling, and $\gamma:\partial P\to\partial Q$ maps distinct facets on disjoint sets of facets.
The tiling $Q$ is called a \emph{macro-tile}, and if $f$ is the $k$-th facet of $P$, then $\gamma(f)$ is called the $k$-th \emph{macro-facet} of $Q$.
\end{definition}

\noindent Fig.~\ref{fig:rauzy_rules_a} illustrates this definition.

\medskip
\begin{figure}[hbtp]
\centering
\includegraphics[width=\textwidth]{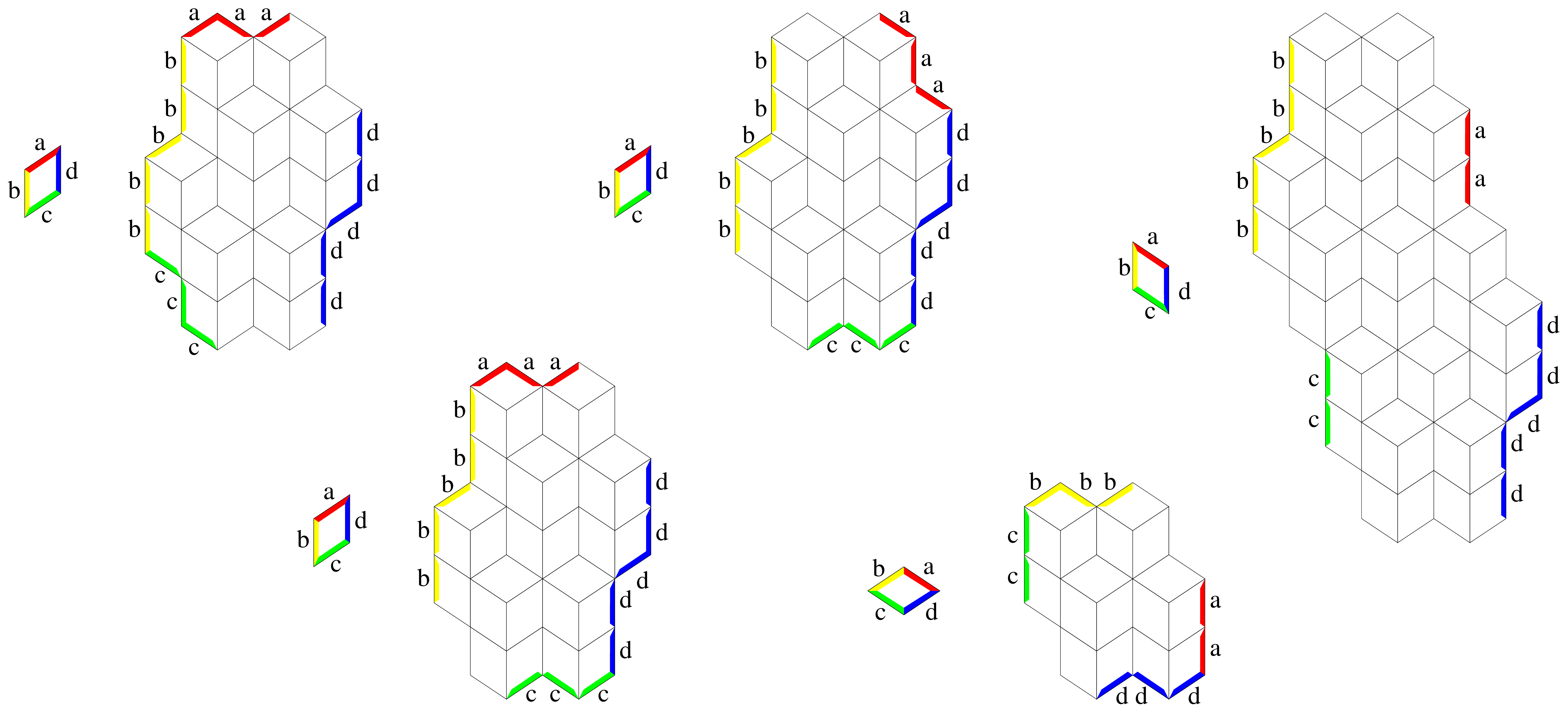}
\caption{These five rules define the so-called Rauzy combinatorial substitution (a facet $f$ and the corresponding macro-facet $\gamma(f)$ are similarly marked).}
\label{fig:rauzy_rules_a}
\end{figure}

We call \emph{tiling by macro-tiles} a tiling whose tiles can be partitioned into macro-tiles, with each macro-facet belonging to the intersection of exactly two macro-tiles.
We can associate with each combinatorial substitution a binary relation over tilings:

\begin{definition}\label{def:preimage}
Let $\sigma$ be a combinatorial substitution.
A tiling $T$ by tiles of $\sigma$ and a tiling $T'$ by macro-tiles of $\sigma$ are said to be \emph{$\sigma$-related} if there is a one-to-one correspondence between the tiles of $T$ and the macro-tiles of $T'$ which preserves the combinatorial structure, that is, such that the $a$-th facet of a first tile of $T$ matches the $b$-th facet of a second tile of $T$ if and only if the $a$-th macro-facet of the first corresponding macro-tile of $T'$ matches the $b$-th macro-facet of the second corresponding macro-tile of $T'$.
One calls $T$ a \emph{preimage} of $T'$ and $T'$ an \emph{image} of $T$.
\end{definition}

\medskip
For example, a Rauzy tiling can be uniquely seen as a tiling by Rauzy macro-tiles (Fig.~\ref{fig:rauzy_tiling_b}) and has a unique preimage, which turns out to be itself a Rauzy tiling (both facts are non-trivial; they follow from results proven in \cite{BF}).

\medskip
\begin{figure}[hbtp]
\centering
\includegraphics[width=\textwidth]{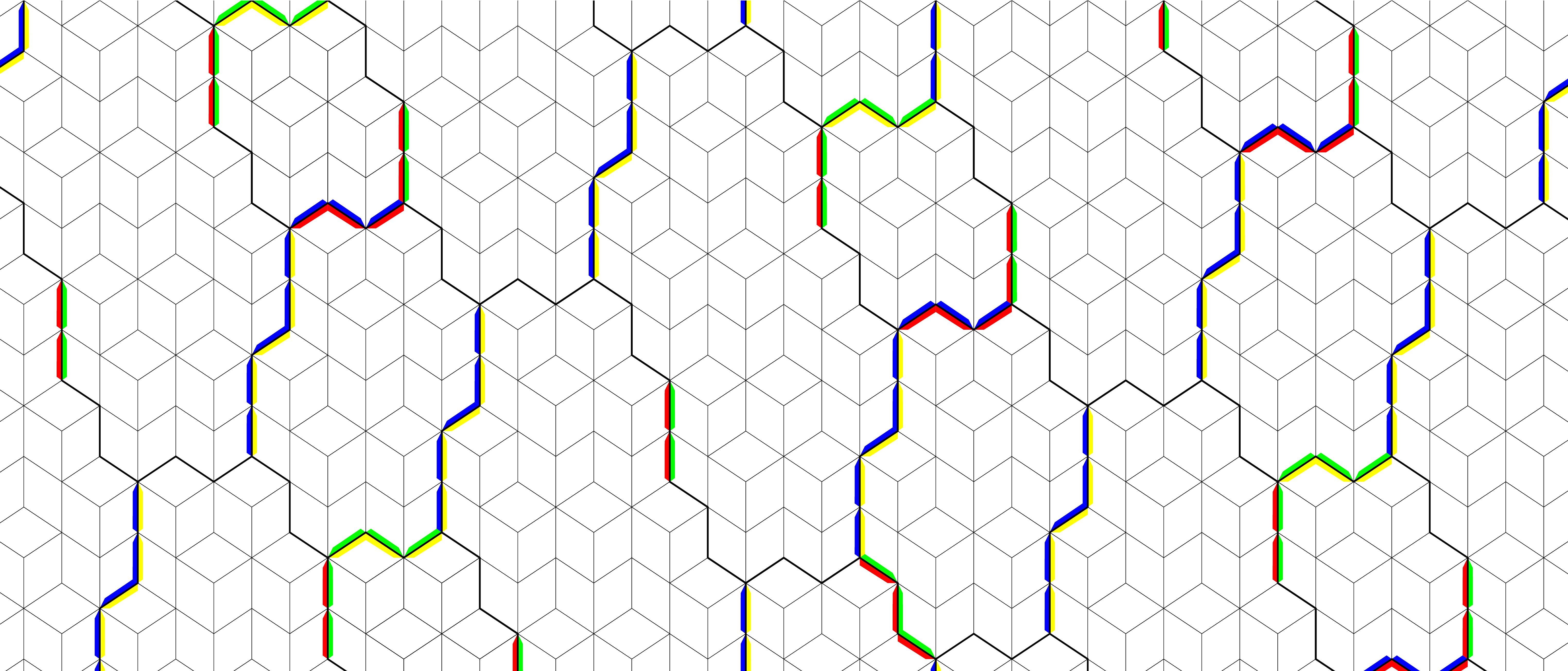}
\caption{A Rauzy tiling can be uniquely seen as a tiling by Rauzy macro-tiles.}
\label{fig:rauzy_tiling_b}
\end{figure}

In particular, the relations associated with combinatorial substitutions yield a strong hierarchical structure on so-called limit-sets:

\begin{definition}\label{def:limit_set}
The \emph{limit set} of a combinatorial substitution $\sigma$, denoted by $\Lambda_\sigma$, is the set of complete tilings which admit an infinite sequence of preimages.
\end{definition}

For example, the limit set of the Rauzy combinatorial substitution is exactly the set of Rauzy tilings (again, this non-trivial fact follows from results proven in \cite{BF}).

\medskip
Let us now turn to the \emph{good combinatorial substitutions} to which Theorem~\ref{th:soficity} applies.
First, a good combinatorial substitution must be \emph{connecting}:

\begin{definition}\label{def:connecting}
A combinatorial substitution $\sigma$ is \emph{connecting} if, for each rule $(P,Q,\gamma)$, the dual graph\footnote{The dual graph of a tiling is the graph whose vertices correspond to tiles of the tiling and whose edges connect vertices corresponding to adjacent tiles.} of $Q$ has a subgraph $N$, called its \emph{network}, such that
\begin{enumerate}
\item $N$ is a star with one branch for each macro-facet, and the leaf of the $k$-th branch is a tile with a facet, called $k$-th \emph{port}, in the $k$-th macro-facet of $Q$;
\item Each macro-facet has non-port facets, and removing the edges of $N$ and its central vertex yields a connected graph which connects\footnote{One says that a subgraph \emph{connects} a set of facets if these facets all belong to tiles which correspond to vertices of this subgraph.} all these facets;
\item the center of $N$ corresponds to a tile in the interior of $Q$, called \emph{central tile};
\item whenever two macro-tiles match along a port, they also match along the corresponding macro-facet.
\end{enumerate}
\end{definition}

Informally, the two first conditions ensure that the macro-facets are big enough to transfer via the network all the informations (encoded by decorations) that we need to enforce the hierarchical structure of the limit set.
In particular, one easily sees that connectivity could not be achieved for combinatorial substitutions on the real line: Theorem \ref{th:soficity} can apply only in dimension two or more.
The third condition ensures that, by iteratively considering macro-tiles of macro-tiles (that is, when going higher and higher in the hierarchy of the limit set), we get tilings covering arbitrarily big balls.
The last condition associates a port with its macro-facet (it is equivalent to the ``sibling-edge-to-edge'' condition of \cite{GS}).

\medskip
\noindent Second, a good combinatorial substitution must be \emph{consistent}:

\begin{definition}\label{def:consistency}
A combinatorial substitution $\sigma$ is said to be \emph{consistent} if any tiling by macro-tiles of $\sigma$ admits a preimage under $\sigma$.
\end{definition}

Intuitively, consistency ensures that, if a tiling meets all the combinatorial conditions to have a preimage, then there is no geometrical obstruction to the existence of such a preimage.

\begin{opb}
Characterize consistent combinatorial substitutions.
\end{opb}

For example, the Rauzy combinatorial substitution is connecting (one easily finds a suitable network for each rule, see Fig.~\ref{fig:rauzy_rules_b} and \ref{fig:rauzy_itere}) and consistent (this non-trivial fact follows from results in \cite{TF}, where explicit maps defined over tilings are shown to be equivalent to such combinatorial substitutions).
Theorem~\ref{th:soficity} thus yields that its limit set, \emph{i.e.}, the set of Rauzy tilings, is sofic.

\begin{figure}[hbtp]
\centering
\includegraphics[width=\textwidth]{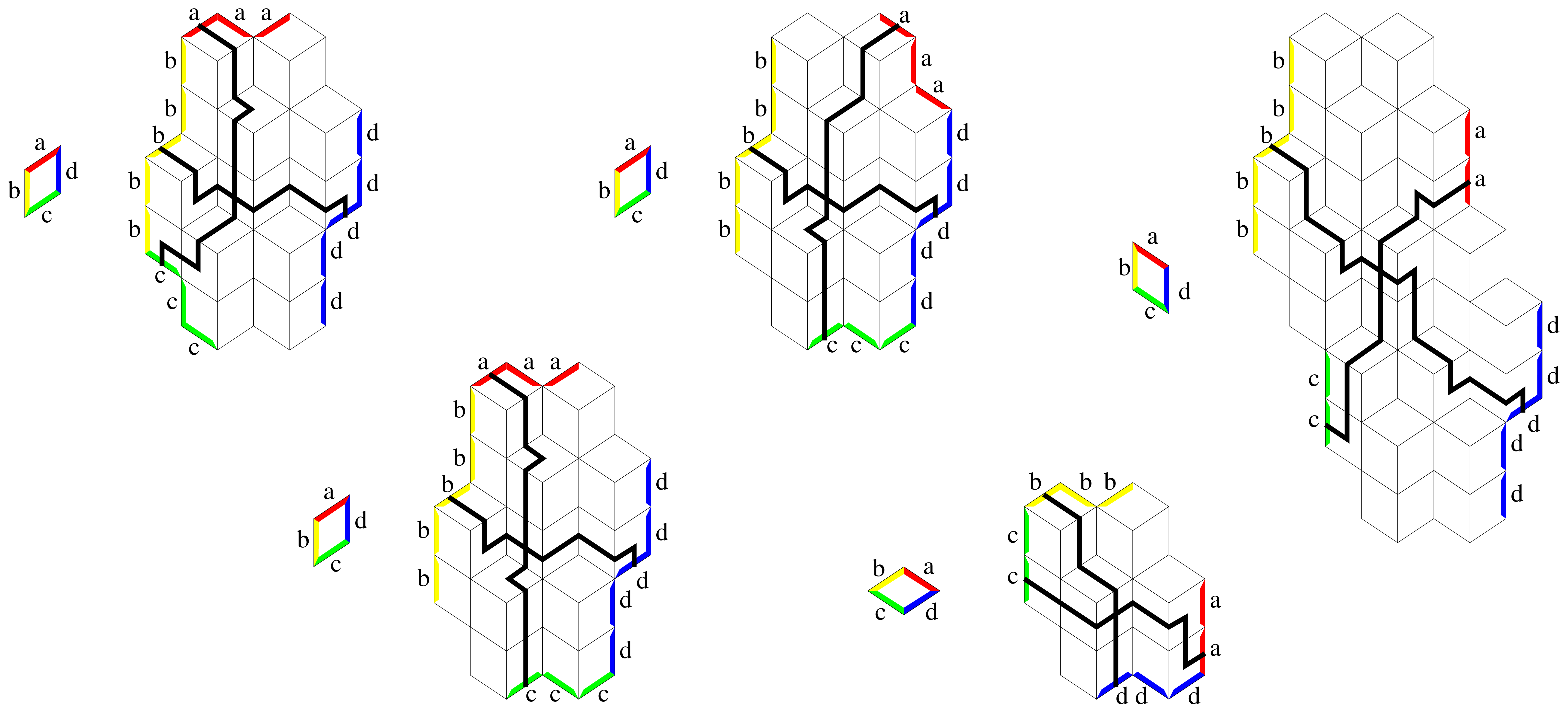}
\caption{The Rauzy combinatorial substitution is connecting.}
\label{fig:rauzy_rules_b}
\end{figure}

\begin{figure}[hbtp]
\centering
\includegraphics[width=0.96\textwidth]{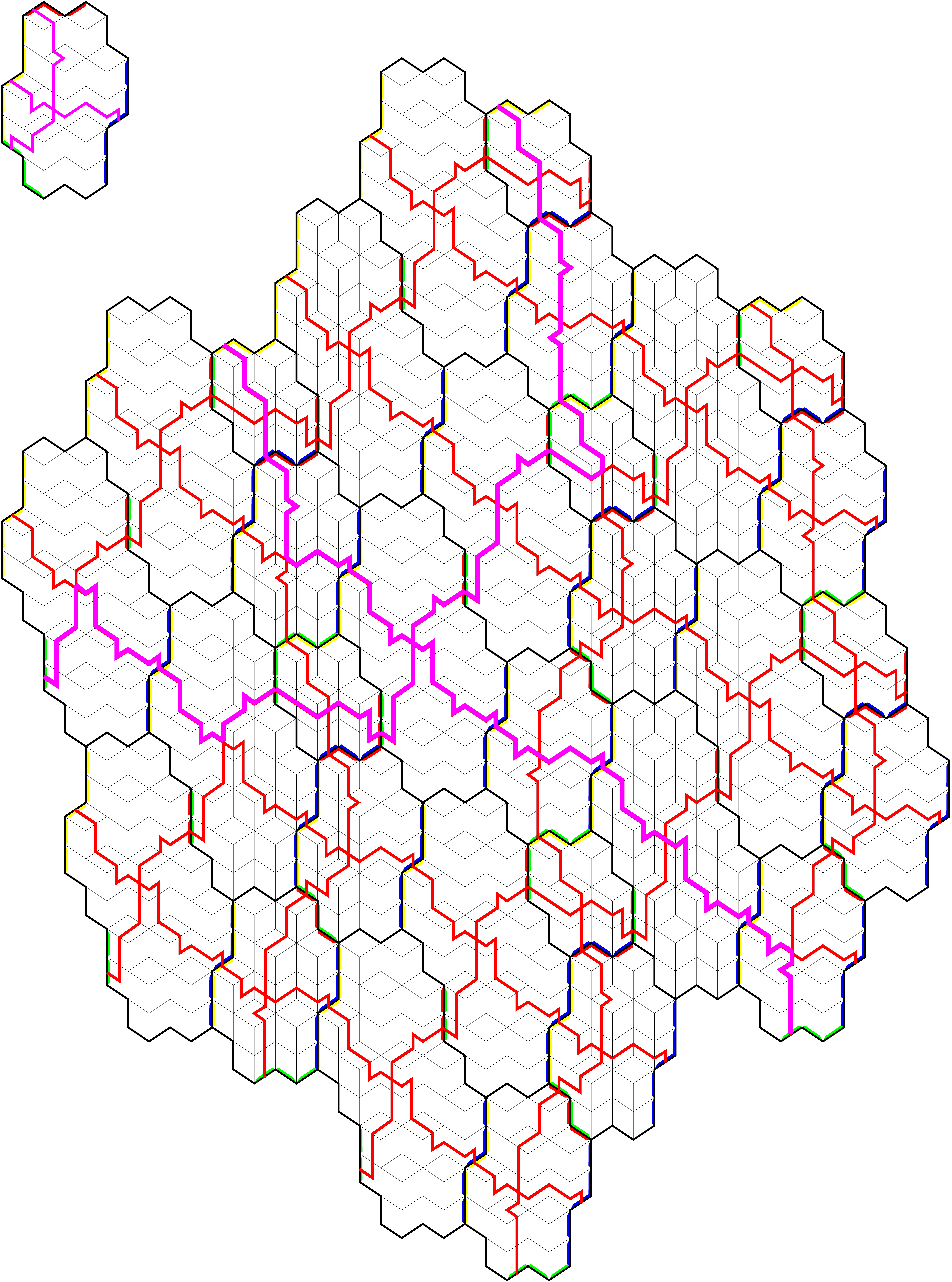}
\caption{A Rauzy macro-tile (top-left) and an image of it under the Rauzy combinatorial substitution, that one could call Rauzy ``macro-macro-tile'' (center).}
\label{fig:rauzy_itere}
\end{figure}

%%%%%%%%%%%%%%%%%%%%%%%
%%%%%%%%%%%%%%%%%%%%%%%
%%%%%%%%%%%%%%%%%%%%%%%
\section{Self-simulation}\label{sec:simulation}

\begin{definition}\label{def:self_simulation}
Let $\sigma$ be a combinatorial substitution with the rules $\{(P_i,Q_i,\gamma_i)\}_i$.
A tileset $\tau$ is said to \emph{$\sigma$-self-simulates} if there is a set of $\tau$-tilings, called $\tau$-macro-tiles, and a map $\phi$ from these $\tau$-macro-tiles into $\tau$ such that
\begin{enumerate}
\item for any $\tau$-macro-tile $\mathcal{Q}$, there is $i$ such that $\pi(\mathcal{Q})=Q_i$ and $\pi(\phi(\mathcal{Q}))=P_i$;
\item any complete $\tau$-tiling is also a tiling by $\tau$-macro-tiles;
\item the $a$-th macro-facet of a $\tau$-macro-tile $\mathcal{Q}$ can match the $b$-th macro-facet of a $\tau$-macro-tile $\mathcal{Q}'$ if and only if the $a$-th facet of the $\tau$-tile $\phi(\mathcal{Q})$ can match the $b$-th facet of the $\tau$-tile $\phi(\mathcal{Q}')$.
\end{enumerate}
\end{definition}

\begin{proposition}\label{prop:self_simulation}
If a tileset $\sigma$-self-simulates for a consistent combinatorial substitution $\sigma$, then its complete tilings are, up to decorations, in the limit set of $\sigma$.
\end{proposition}

\begin{proof}
Consider a complete $\tau$-tiling $\mathcal{P}$.
Conditions (1)--(2) ensure that removing the decorations of $\mathcal{P}$ yields a tiling by macro-tiles of $\sigma$, say $P$.
The consistency of $\sigma$ ensures that $P$ admits a preimage under $\sigma$, say $R$.
Let us show that $R$ can be endowed by decorations to get a $\tau$-tiling.
Consider a tile $T$ of $R$.
This tile corresponds (via the one-to-one correspondence in Def.~\ref{def:preimage}) to a macro-tile $Q$ of $P$, which is itself the image under $\pi$ of a $\tau$-macro-tile $\mathcal{Q}$ of $\mathcal{P}$.
We associate with $T$ the $\tau$-tile $\phi(\mathcal{Q})$.
Condition (3) ensures that replacing tiles in $R$ by their such associated $\tau$-tiles yields a complete $\tau$-tiling, say $\mathcal{R}$.
We can repeat all this process, with $\mathcal{R}$ instead of $\mathcal{P}$.
By induction, we get an infinite sequence of complete tilings, with each one being the preimage under $\sigma$ of the previous one.
Thus, $\mathcal{P}\in\Lambda_\sigma$.
\end{proof}

%%%%%%%%%%%%%%%%%%%%%%%
%%%%%%%%%%%%%%%%%%%%%%%
%%%%%%%%%%%%%%%%%%%%%%%
\section{Constructive proof of Theorem~\ref{th:soficity}}\label{sec:proof}

Let $\sigma$ be a good combinatorial substitution with rules ${(P_i,Q_i,\gamma_i)}_i$.
We here rely on the fact that $\sigma$ is connecting to construct a finite tileset $\tau$ which $\sigma$-self-simulates.
The consistency of $\sigma$ then ensures, via Prop.~\ref{prop:self_simulation}, that $\pi(\Lambda_\tau)\subseteq\Lambda_\sigma$ holds.
We also show that the converse inclusion holds.
This thus constructively proves Theorem~\ref{th:soficity}.

\subsection{Settings}
Let $T_1,\ldots,T_n$ and $f_1,\ldots,f_m$ be numberings of, respectively, the tiles and the internal facets of all the $Q_i$'s.
Given the $k$-th facet of the tile $T_i$, $N_\sigma(i,k)$ stands either for its index if it is an internal facet, or for a special value ``port'', ``macro-facet'' or ``boundary'' otherwise (depending whether it is a port, a non-port facet in a macro-facet or another external facet).

\medskip
Each tile of $\tau$ is a $T_i$ endowed with a decoration which encodes on each facet a triple $(f,j,g)$, where $f$ and $g$ are either facet indices or special values ``port'', ``macro-facet'' or ``boundary'', and $j$ is either zero or a tile index.
We call $f$ the \emph{macro-index}, $j$ the \emph{parent-index} and $g$ the \emph{neighbor-index}.
This clearly allows only a finite number of different tiles.

\medskip
We skip the technical details concerning the way these triples are encoded by decorations\footnote{Recall that the decoration of a tile is a real map defined on its boundary.}.
We assume that two decorated tiles match along a facet if and only if the same triple is encoded on both facets, and that the only direct isometry which leaves invariant the decoration of a facet is the identity, so that a decorated tile cannot trivially match with a translated or rotated copy of itself.

\subsection{Decorations}
\noindent The five following steps completely define a tileset $\tau$.

\medskip
\noindent{\bf 1.}
The macro-index of the $k$-th facet of any decorated $T_j$ is $N_\sigma(j,k)$.
This step ensures that any complete $\tau$-tiling can be uniquely seen as a tiling by $\tau$-macro-tiles.

\medskip
\noindent{\bf 2.}
Consider a decorated non-central tile of $Q_i$.
Its facets which are internal and not crossed by the network have all the same parent-index, also called parent-index of the tile, which can be any $j$ such that $T_j=P_i$.
Its facets which are external, port excluded, have parent-index $0$.
This step ensures that the $\tau$-tiles of a $\tau$-macro-tile $\mathcal{Q}$ share a common parent-index $j$; the tile $T_j$ is called the \emph{parent-tile} of $\mathcal{Q}$.

\medskip
\noindent{\bf 3.}
Consider, in a non-central $\tau$-tile with parent-index $j$, a facet which is neither a port nor crossed by the network.
Its neighbor-index is either $N_\sigma(j,k)$ if it is in the $k$-th macro-facet, or equal to its macro-index otherwise.
This step ensures that the macro-facets (ports excepted) of a $\tau$-macro-tile are equivalent to the macro-indices of its parent tile (once decorated).

\medskip
\noindent{\bf 4.}
Consider a non-central $\tau$-tile with parent-index $j$.
Its facets which are either $k$-th port or crossed by the $k$-th branch of the network have all the same pair of parent/neighbor indices: it can be any pair not forbidden on the $k$-th facet of $T_j$ (Steps 1--3).
%This pair is either one of the pairs allowed on the $k$-th facet of $T_j$, if those are already defined in Steps~1--3 (\emph{i.e.}, if the $k$-th facet of $T_j$ is not crossed by a network), or any of the pairs defined in Steps~1--3 otherwise.
This step ensures that each port of a $\tau$-macro-tile is equivalent to the pair of parent/neighbor indices of a decorated parent-tile\footnote{With the problem that a parent-tile can be decorated in different ways, and nothing yet prevents the ports from mixing these decorations.}.

\medskip
\noindent{\bf 5.}
Whenever a non-central $\tau$-tile $\mathcal{T}$ has as many facets as a central tile $T_j$, we define a central $\tau$-tile $\mathcal{T}'$ by endowing each $k$-th facet of $T_j$ with the pair of parent/neighbor indices on the $k$-th facet of $\mathcal{T}$ (the macro-indices are defined as usual, see Step~1).
One says that $\mathcal{T}'$ \emph{derives} from $\mathcal{T}$.

\subsection{First inclusion}
Let us show that the above defined tileset $\sigma$-self-simulates.
Given a $\tau$-macro-tile $\mathcal{Q}$ with parent-index $j$ and central $\tau$-tile $\mathcal{T}'$, let $\phi(\mathcal{Q})$ be the decorated tile obtained by endowing the $k$-th facet of $T_j$ with the parent/neighbor indices on the $k$-th facet of $\mathcal{T}'$ (the macro-indices are defined as for any tile, see Step 1).
One checks that $\phi$ is a map satisfying conditions (1)--(3) of Def. \ref{def:self_simulation}.
It remains to check that $\phi(\mathcal{Q})\in\tau$.

\medskip
Let $\mathcal{T}$ be a non-central $\tau$-tile from which derives $\mathcal{T}'$.
If $T_j$ is central, then $\phi(\mathcal{Q})$ also derives from $\mathcal{T}$, hence is in $\tau$.
Otherwise, Step 4 ensures, for each $k$, that the parent/neighbor indices on the $k$-th facet of $\mathcal{T}'$ appear on the $k$-th facet of a decorated $T_j$.
In particular, this holds on facets of $T_j$ which are not crossed by a network.
Since the neighbor-indices of these facets can only\footnote{Actually, a facet-index appears on the two tiles of a macro-tile which share the corresponding facet.
We thus need, in order to completely characterize $T_j$, either to assume that there is at least two such facets (this is a rather mild assumption), or to endow facets with an \emph{orientation} and to allow two facets to match if and only if they have opposite orientations.} appear on a $T_j$ (see Steps~1 and 3), $\mathcal{T}$ is a decorated $T_j$.
This yields $\phi(\mathcal{Q})=\mathcal{T}$, and thus $\phi(\mathcal{Q})$ is in $\tau$.

\medskip
\noindent Prop.~\ref{prop:self_simulation} then applies and yields the first inclusion $\pi(\Lambda_\tau)\subseteq\Lambda_\sigma$.

\subsection{Second inclusion}
Let us extend $\tau$ in a tileset $\tau'$ as follow.
For each $\tau$-tile $\mathcal{T}$ and each subset $S$ of its facets, we define a $\tau'$-tile $\mathcal{T}_S$ by replacing the decorations of the facets in $S$ by a special decoration ``undefined''.

\medskip
Now, let $P$ be a tiling in $\Lambda_\sigma$.
Consider an infinite sequence $(P_n)_{n\geq 0}$ of successive preimages of $P$.
Given $n>0$, $P_n$ can be seen as a $\tau'$-tiling: it suffices to endow any facet with ``undefined''.
Then, $P_{n-1}$ can be seen as a tiling by $\tau'$-macro-tiles, with ``undefined'' decorations appearing only on the network.
Indeed, consider a macro-tile of $P_{n-1}$ which corresponds (via the one-to-one correspondence in Def.~\ref{def:preimage}) to a tile $T_j$ in $P_n$: it suffices to endow its tiles as in steps~1--3 of the definition of $\tau$, with $T_j$ being the parent-tile, and with the decoration ``undefined'' on the facets crossed by the network.
This can be iterated up to $P=P_0$, and the third condition of Def.~\ref{def:connecting} ensures that the decorations ``undefined'' appear only on sort of grids whose cells have bigger and bigger size.

\medskip
Thus, by making $n$ tend to infinity, one can see $P$ as a $\tau'$-tiling whose ``undefined'' decorations, if any, form either a star with $k$ infinite branches, or a single biinfinite branch.
In the first case, we can replace the central $\tau'$-tile by any $\tau$-tile with $k$ facets, and then  the tiles on branches by $\tau$-tiles which carry decorations to infinity.
In the second case, we can replace the $\tau'$-tiles by $\tau$-tiles which carry any decoration on the whole branch.
In any case, we can thus see $P$ as a $\tau$-tiling.

\medskip
\noindent We thus have the second inclusion $\Lambda_\sigma\subseteq\pi(\Lambda_\tau)$.
Both inclusions prove Theorem~\ref{th:soficity}.

%%%%%%%%%%%%%%%%%%%%%%%
%%%%%%%%%%%%%%%%%%%%%%%
%%%%%%%%%%%%%%%%%%%%%%%
\section{On the number of tiles}
\label{sec:counting}

Let us conclude this paper by discussing the size $\#\tau$ of the tileset $\tau$ defined in the previous section (although its finiteness suffices for Theorem~\ref{th:soficity}).

\medskip
Consider a good combinatorial substitution with $r$ rules.
Fix a network as in Def.~\ref{def:connecting} for each of these rules.
Let $n$ and $m$ denote the total number of, respectively, tiles and internal facets of the macro-tiles of these $r$ rules.
Among these $n$ tiles, let $p$ denote the number of those which lie on a network.

\medskip
First, the $n-p$ tiles not on the network can be decorated in at most $n$ different ways, according to the parent-index they carry.
This yields at most $N_0=(n-p)n$ $\tau$-tiles.
Then, each of the $p-r$ non-central tiles on the network can be decorated in at most $(n-p)n+pnm$ different ways: $n-p$ parent-indices which correspond to tiles not on the network, hence allow at most $n$ pairs of parent/neighbor indices carried on the network, and $p$ parent-indices which correspond to tiles on the network, hence allow at most $nm$ pairs of parent/neighbor indices carried on the network.
This yields at most $N_p=(p-r)((n-p)n+pmn)$ $\tau$-tiles.
Last, the $r$ central tiles can be decorated in at most as many different ways as there is non-central $\tau$-tiles.
Finally, this yields
$$
\#\tau\leq (r+1)(N_0+N_p)\leq (r+2)p^2mn.
$$
This bound is, for example, about one billion in the Rauzy case.

\medskip
\begin{figure}[hbtp]
\centering
\includegraphics[width=\textwidth]{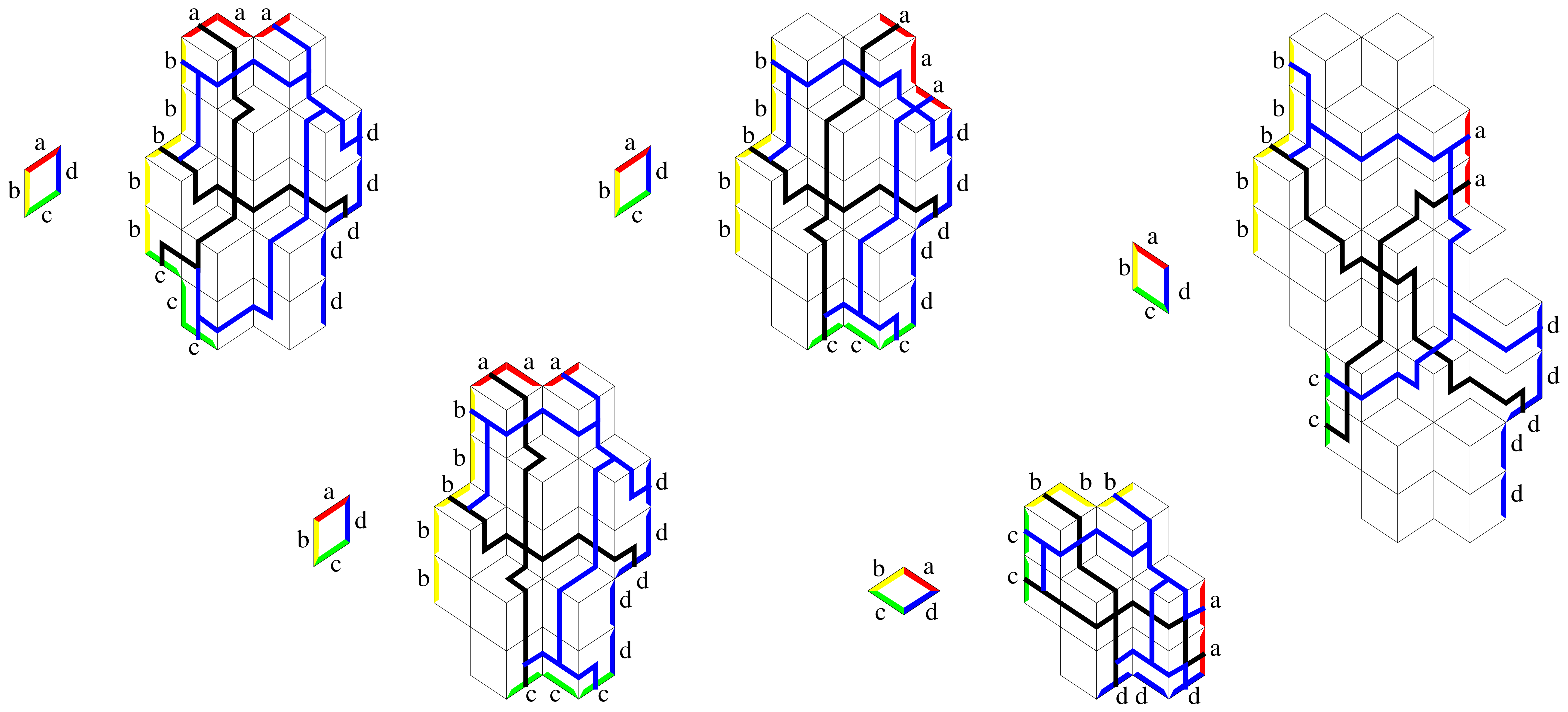}
\caption{A second network for the Rauzy combinatorial substitution.}
\label{fig:rauzy_rules_c}
\end{figure}

\medskip
In order to reduce this huge number of tiles, it is worth noting that, instead of carrying a parent-index through all the tiles of a macro-tile, it suffices to carry it along a \emph{second network} connecting the macro-facets of this macro-tile and intersecting each of the branches of its (first) network (see, \emph{e.g.}, Fig. \ref{fig:rauzy_rules_c}).
The ``control'' described in Step~4 is then performed only on tiles where both networks crosses, while we simply allow any possible pairs of parent/neighbor indices to be carried through the tiles which are only on the first network.
If we denote by $q$ the total number of tiles on these new second networks and by $c$ be the total number of crossings between second and first networks, a similar analysis yields
$$
\#\tau \leq (r+1)(N'_0+N'_q+N'_p+N'_c) \leq (r+2)cp(m+qn),
$$
where:\\
\begin{minipage}{0.5\textwidth}
\begin{eqnarray*}
N'_0&=&n-p-q+c,\\
N'_q&=&(q-c)n,\\
\end{eqnarray*}
\end{minipage}
\begin{minipage}{0.5\textwidth}
\begin{eqnarray*}
N'_p&=&(p-c)(m+qn),\\
N'_c&=&c(N'_0+N'_q+N'_p).\\
\end{eqnarray*}
\end{minipage}
This bound is, for example, about $70$ millions tiles in the Rauzy case.

\medskip
This last bound is huge but generic.
One can hope to dramatically decrease this bound in specific cases.
Indeed, most of the $\tau$-tiles correspond to tiles which simply carry any possible information (the tiles on networks, crossings excepted).
Since these tiles all play the same role, it would be worth to replace all the tiles on a network by a single tile (one can thus hope to gain a factor $pq$ in the above bound -- this would yield about $25000$ tiles in the Rauzy case).
This shall however be done carefully, so that tiles still necessarily form macro-tiles.
Note that it should be much easier in dimension $d\geq 3$, since the cohesion of macro-tiles can be more easily enforced without relying on tiles on networks.

\bigskip
\textbf{Acknowledgments}. We would like to thank the anonymous referees for their valuable comments. The first author also would like to thank N. Pytheas Fogg and C. Goodman-Strauss for their encouragements to write down this result, as well as for useful discussions.

\providecommand{\bysame}{\leavevmode\hbox to3em{\hrulefill}\thinspace}
\providecommand{\MR}{\relax\ifhmode\unskip\space\fi MR }
\providecommand{\MRhref}[2]{%
  \href{http://www.ams.org/mathscinet-getitem?mr=#1}{#2}
}
\providecommand{\href}[2]{#2}

\clearpage\appendix
\section{Complete example: the $3x3$ square substitution}

\tikzset{coolgrid/.style={dash pattern=on 0.2cm off 0.6cm on 0.2cm off 0cm, line cap=rect}}
\colorlet{ftile}{yellow!10}
\colorlet{dual}{magenta!80!black}
\colorlet{ndual}{green!80!black}
\newcommand\cog[1]{\color{cyan!80!black}\textsf{#1}}
\newcommand\coj[1]{\color{red!80!black}\textsf{#1}}
\newcommand\cof[1]{\color{green!80!black}\textsf{#1}}
\newenvironment{blingbling}{\mbox\bgroup}{\egroup}
\pgfkeys{
/fjg/.cd,
f/.store in = \daf,
j/.store in = \daj,
g/.store in = \dag
}
\newcommand\clearfjg{\gdef\daf{\color{black!70}?}\gdef\daj{\color{black!70}?}\gdef\dag{\color{black!70}?}}
\newcommand*\setfjg{\pgfqkeys{/fjg}}
\newcommand\decotile[6]{\begin{scope}[shift={(#6)}]
\begin{scope}
\filldraw[fill=ftile] (0,0) rectangle ++(3,3);
\clip (0,0) rectangle ++(3,3);
\draw (0,0) -- ++(1,1) -- ++(0,1) -- ++(-1,1) -- cycle;
\draw (0,0) -- ++(1,1) -- ++(1,0) -- ++(1,-1) -- cycle;
\draw (3,3) -- ++(-1,-1) -- ++(0,-1) -- ++(1,-1) -- cycle;
\draw (3,3) -- ++(-1,-1) -- ++(-1,0) -- ++(-1,1) -- cycle;
\end{scope}
\clearfjg\setfjg{#1}%
\path (0.8,0) node [above] {\cof{\daf}}
	  ++(0.7,0) node [above] {\coj{\daj}}
	  ++(0.7,0) node [above] {\cog{\dag}};
\clearfjg\setfjg{#2}%
\path (0.8,3) node [below] {\vphantom{$f$}\cof{\daf}}
	  ++(0.7,0) node [below] {\vphantom{$f$}\coj{\daj}}
	  ++(0.7,0) node [below] {\vphantom{$f$}\cog{\dag}};
\clearfjg\setfjg{#3}%
\path (0,0.8) node [right] {\cog{\dag}}
	  ++(0,0.7) node [right] {\coj{\daj}}
	  ++(0,0.7) node [right] {\cof{\daf}};
\clearfjg\setfjg{#4}%
\path (3,0.8) node [left] {\cog{\dag}}
	  ++(0,0.7) node [left] {\coj{\daj}}
	  ++(0,0.7) node [left] {\cof{\daf}};
\path (1.5,1.5) node {\large #5};
\end{scope}}

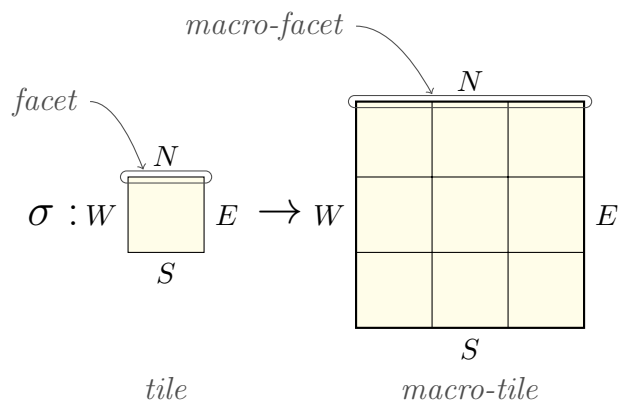
\begin{figure}[!ht]
\centering
	%% \sigma
	\begin{blingbling}
	\begin{tikzpicture}
	\begin{scope}
	\filldraw[fill=ftile] (0,0) rectangle ++(1,1);
	\path (0,0.5) node[left] {\small $W$}
	      (1,0.5) node[right] {\small $E$}
	      (0.5,1) node[above] {\small $N$}
	      (0.5,0) node[below] {\small $S$};
	\draw[black!70,rounded corners=0.8mm] (0,1) ++(-0.1,-0.08) rectangle ++(1.2,0.16);
	\draw[->,black!70] (-0.5,2) node[left] {\emph{facet}} parabola (0.2,1.1);
	\end{scope}
	\path (-0.5,0.5) node[left] {\Large $\sigma :$}
	      (2,0.5) node {\Large $\rightarrow$};
	\begin{scope}[shift={(3,-1)}]
	\filldraw[fill=ftile,thick] (0,0) rectangle ++(3,3);
	\draw (0,0) grid ++(3,3);
	\path (0,1.5) node[left] {\small $W$}
	      (3,1.5) node[right] {\small $E$}
	      (1.5,3) node[above] {\small $N$}
	      (1.5,0) node[below] {\small $S$};
	\draw[black!70,rounded corners=0.8mm] (0,3) ++(-0.1,-0.08) rectangle ++(3.2,0.16);
	\draw[->,black!70] (0,4) node[left] {\emph{macro-facet}} parabola (1,3.1);
	\end{scope}
	\path (0.5,-1.5) node[below,black!70] {\emph{tile}}
		  (4.5,-1.5) node[below,black!70] {\emph{macro-tile}};
	\end{tikzpicture}%
	\end{blingbling}
\caption{a simple substitution}
\end{figure}

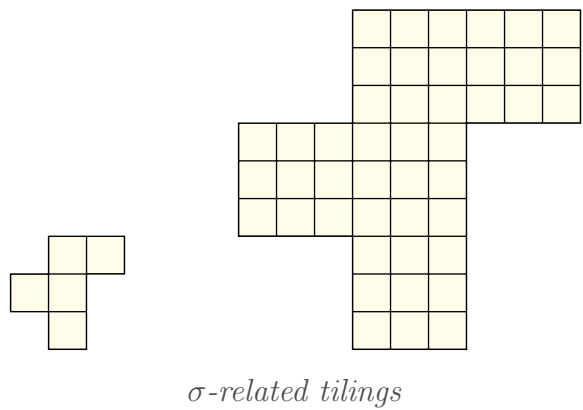
\begin{figure}[!ht]
\centering
	%%% sigma related tilings 
	\begin{blingbling}
	\begin{tikzpicture}[line width=0.5pt,scale=0.5]
	\begin{scope}
	\filldraw[fill=ftile] (0,0) -- ++(1,0) -- ++(0,2) -- ++(1,0) -- ++(0,1) -- ++(-2,0) -- ++(0,-1) -- ++(-1,0) -- ++(0,-1) -- ++(1,0) -- cycle;
	\clip (0,0) -- ++(1,0) -- ++(0,2) -- ++(1,0) -- ++(0,1) -- ++(-2,0) -- ++(0,-1) -- ++(-1,0) -- ++(0,-1) -- ++(1,0) -- cycle;
	\draw (-1,0) grid ++(3,3);
	\end{scope}	
	\begin{scope}[shift={(8,0)}]
	\filldraw[fill=ftile,scale=3] (0,0) -- ++(1,0) -- ++(0,2) -- ++(1,0) -- ++(0,1) -- ++(-2,0) -- ++(0,-1) -- ++(-1,0) -- ++(0,-1) -- ++(1,0) -- cycle;
	\clip[scale=3] (0,0) -- ++(1,0) -- ++(0,2) -- ++(1,0) -- ++(0,1) -- ++(-2,0) -- ++(0,-1) -- ++(-1,0) -- ++(0,-1) -- ++(1,0) -- cycle;
	\draw (-3,0) grid ++(9,9);
	\end{scope}
	\path (6.5,-0.5) node[below,black!70] {\emph{$\sigma$-related tilings}};
	\end{tikzpicture}%
	\end{blingbling}
\caption{sample $\sigma$-related tilings}
\end{figure}

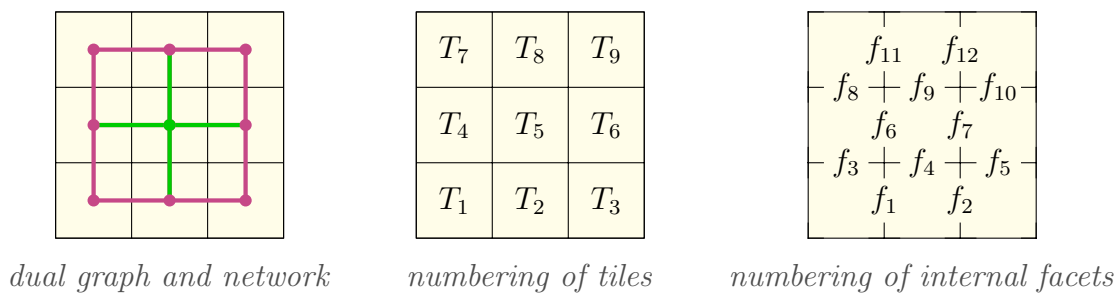
\begin{figure}[!ht]
\centering
	%%% N
	\begin{blingbling}
	\begin{tikzpicture}
	\filldraw[fill=ftile] (0,0) rectangle ++(3,3);
	\draw (0,0) grid ++(3,3);
	\draw[shift={(0.5,0.5)},ultra thick,dual] (0,0) grid ++(2,2);
	\draw[ultra thick,ndual] (0.5,1.5) -- ++(2,0);
	\draw[ultra thick,ndual] (1.5,0.5) -- ++(0,2);
	\foreach \x in {0,1,2}
	\foreach \y in {0,1,2}
	\fill[dual] (\x+0.5,\y+0.5) circle (0.08);
	\fill[ndual] (1.5,1.5) circle (0.08);
	\path (1.5,-0.2) node[below,black!70] {\emph{dual graph and network}};
	\end{tikzpicture}%	
	\end{blingbling}\hfill
	%%% tiles numbering
	\begin{blingbling}
	\begin{tikzpicture}
	\filldraw[fill=ftile] (0,0) rectangle ++(3,3);
	\draw (0,0) grid ++(3,3);
	\path[shift={(0.5,0.5)}] (0,0) node {$T_1$} (1,0) node {$T_2$} (2,0) node {$T_3$} (0,1) node {$T_4$} (1,1) node {$T_5$} (2,1) node {$T_6$} (0,2) node {$T_7$} (1,2) node {$T_8$} (2,2) node {$T_9$};
	\path (1.5,-0.2) node[below,black!70] {\emph{numbering of tiles}};
	\end{tikzpicture}%	
	\end{blingbling}\hfill
	%%% facets numbering
	\begin{blingbling}
	\begin{tikzpicture}
	\filldraw[fill=ftile] (0,0) rectangle ++(3,3);
	\draw[coolgrid] (0,0) grid ++(3,3);
	\path[shift={(0,0.5)}] (1,0) node {$f_1$} (2,0) node {$f_2$} (1,1) node {$f_6$} (2,1) node {$f_7$} (1,2) node {$f_{11}$} (2,2) node {$f_{12}$};
	\path[shift={(0.5,0)}] (0,1) node {$f_3$} (1,1) node {$f_4$} (2,1) node {$f_5$} (0,2) node {$f_8$} (1,2) node {$f_9$} (2,2) node {$f_{10}$};
	\path (1.5,-0.2) node[below,black!70] {\emph{numbering of internal facets}};
	\end{tikzpicture}%	
	\end{blingbling}
\caption{chosen network and numberings}
\end{figure}

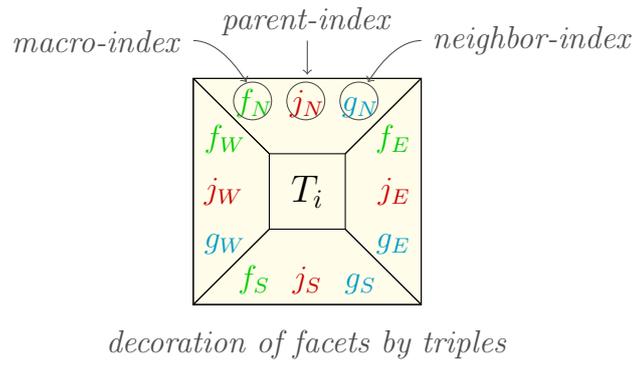
\begin{figure}[!ht]
\centering
	%%% decorations
	\begin{blingbling}
	\begin{tikzpicture}
	\decotile{f={$f_S$},j={$j_S$},g={$g_S$}}{f={$f_N$},j={$j_N$},g={$g_N$}}{f={$f_W$},j={$j_W$},g={$g_W$}}{f={$f_E$},j={$j_E$},g={$g_E$}}{$T_i$}{0,0}
	\draw[black!70] (0.78,2.7) circle (0.25);
	\draw[->,black!70] (0,3.5) node[left] {\emph{macro-index}} parabola (0.7,2.95);
	\draw[black!70] (1.48,2.7) circle (0.25);
	\draw[->,black!70] (1.5,3.5) node[above=-2pt] {\emph{parent-index}} -- (1.5,3.05);
	\draw[black!70] (2.18,2.7) circle (0.25);
	\draw[->,black!70] (3,3.5) node[right] {\emph{neighbor-index}} parabola (2.3,2.95);
	\path (1.5,-0.2) node[below,black!70] {\emph{decoration of facets by triples}};
	\end{tikzpicture}%
	\end{blingbling}
\caption{decoration of facets}
\end{figure}

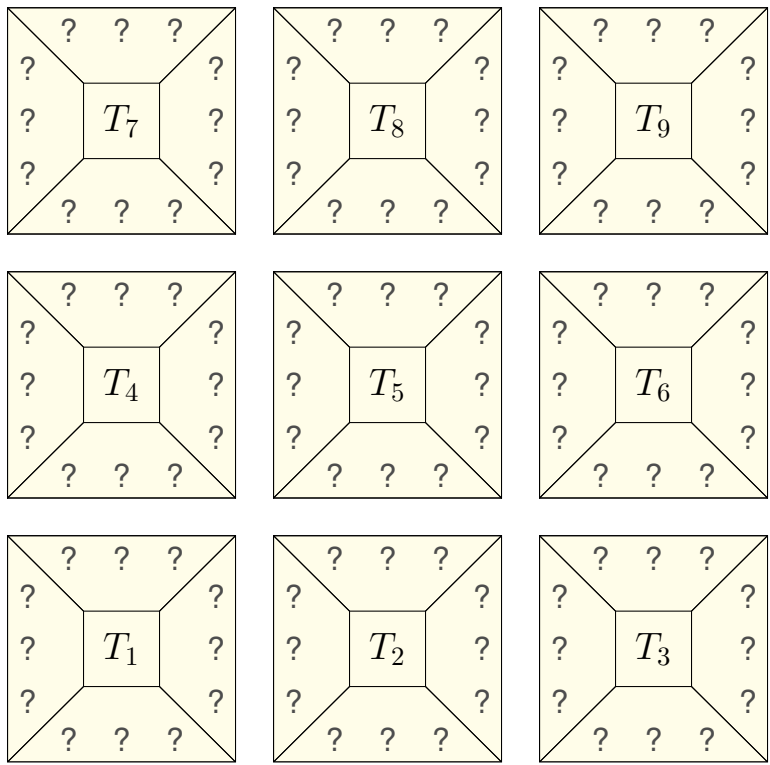
\begin{figure}[!ht]
\centering
	%%% step 0
	\begin{blingbling}
	\begin{tikzpicture}
	\decotile{}{}{}{}{$T_1$}{0,0}
	\decotile{}{}{}{}{$T_2$}{3.5,0}
	\decotile{}{}{}{}{$T_3$}{7,0}
	\decotile{}{}{}{}{$T_4$}{0,3.5}
	\decotile{}{}{}{}{$T_5$}{3.5,3.5}
	\decotile{}{}{}{}{$T_6$}{7,3.5}
	\decotile{}{}{}{}{$T_7$}{0,7}
	\decotile{}{}{}{}{$T_8$}{3.5,7}
	\decotile{}{}{}{}{$T_9$}{7,7}
	\end{tikzpicture}%
	\end{blingbling}
\caption{initial decoration scheme}
\end{figure}

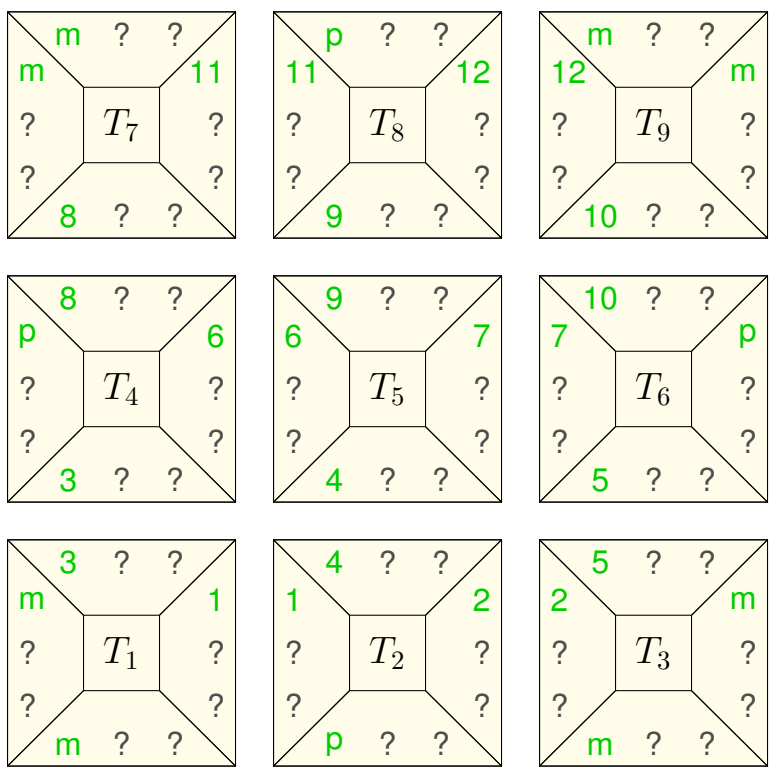
\begin{figure}[!ht]
\centering
	%%% step 1
	\begin{blingbling}
	\begin{tikzpicture}
	\decotile{f={m}}{f={3}}{f={m}}{f={1}}{$T_1$}{0,0}
	\decotile{f={p}}{f={4}}{f={1}}{f={2}}{$T_2$}{3.5,0}
	\decotile{f={m}}{f={5}}{f={2}}{f={m}}{$T_3$}{7,0}
	\decotile{f={3}}{f={8}}{f={p}}{f={6}}{$T_4$}{0,3.5}
	\decotile{f={4}}{f={9}}{f={6}}{f={7}}{$T_5$}{3.5,3.5}
	\decotile{f={5}}{f={10}}{f={7}}{f={p}}{$T_6$}{7,3.5}
	\decotile{f={8}}{f={m}}{f={m}}{f={11}}{$T_7$}{0,7}
	\decotile{f={9}}{f={p}}{f={11}}{f={12}}{$T_8$}{3.5,7}
	\decotile{f={10}}{f={m}}{f={12}}{f={m}}{$T_9$}{7,7}
	\end{tikzpicture}%
	\end{blingbling}
\caption{after step 1: macro-indices are fixed}
\end{figure}

\begin{figure}[!ht]
\centering
	%%% step 2
	\begin{blingbling}
	\begin{tikzpicture}
	\decotile{f={m},j={0}}{f={3},j={$j$}}{f={m},j={0}}{f={1},j={$j$}}{$T_1$}{0,0}
	\decotile{f={p}}{f={4}}{f={1},j={$j$}}{f={2},j={$j$}}{$T_2$}{3.5,0}
	\decotile{f={m},j={0}}{f={5},j={$j$}}{f={2},j={$j$}}{f={m},j={0}}{$T_3$}{7,0}
	\decotile{f={3},j={$j$}}{f={8},j={$j$}}{f={p}}{f={6}}{$T_4$}{0,3.5}
	\decotile{f={4}}{f={9}}{f={6}}{f={7}}{$T_5$}{3.5,3.5}
	\decotile{f={5},j={$j$}}{f={10},j={$j$}}{f={7}}{f={p}}{$T_6$}{7,3.5}
	\decotile{f={8},j={$j$}}{f={m},j={0}}{f={m},j={0}}{f={11},j={$j$}}{$T_7$}{0,7}
	\decotile{f={9}}{f={p}}{f={11},j={$j$}}{f={12},j={$j$}}{$T_8$}{3.5,7}
	\decotile{f={10},j={$j$}}{f={m},j={0}}{f={12},j={$j$}}{f={m},j={0}}{$T_9$}{7,7}
	\end{tikzpicture}%
	\end{blingbling}
\caption{after step 2: parent-indices are fixed outside network}
\end{figure}
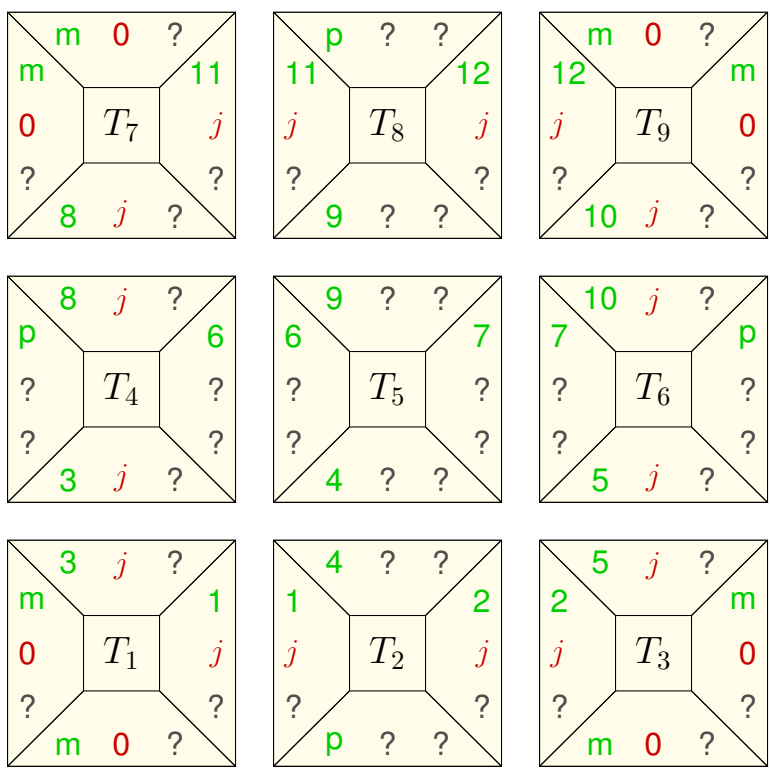

\begin{figure}[!ht]
\centering
	%%% step 3
	\begin{blingbling}
	\begin{tikzpicture}
	\decotile{f={m},j={0},g={$s$}}{f={3},j={$j$},g={3}}{f={m},j={0},g={$w$}}{f={1},j={$j$},g={1}}{$T_1$}{0,0}
	\decotile{f={p}}{f={4}}{f={1},j={$j$},g={1}}{f={2},j={$j$},g={2}}{$T_2$}{3.5,0}
	\decotile{f={m},j={0},g={$s$}}{f={5},j={$j$},g={5}}{f={2},j={$j$},g={2}}{f={m},j={0},g={$e$}}{$T_3$}{7,0}
	\decotile{f={3},j={$j$},g={3}}{f={8},j={$j$},g={8}}{f={p}}{f={6}}{$T_4$}{0,3.5}
	\decotile{f={4}}{f={9}}{f={6}}{f={7}}{$T_5$}{3.5,3.5}
	\decotile{f={5},j={$j$},g={5}}{f={10},j={$j$},g={10}}{f={7}}{f={p}}{$T_6$}{7,3.5}
	\decotile{f={8},j={$j$},g={8}}{f={m},j={0},g={$n$}}{f={m},j={0},g={$w$}}{f={11},j={$j$},g={11}}{$T_7$}{0,7}
	\decotile{f={9}}{f={p}}{f={11},j={$j$},g={11}}{f={12},j={$j$},g={12}}{$T_8$}{3.5,7}
	\decotile{f={10},j={$j$},g={10}}{f={m},j={0},g={$n$}}{f={12},j={$j$},g={12}}{f={m},j={0},g={$e$}}{$T_9$}{7,7}
	\path (5,-0.2) node[below,black!70] {\emph{where $n=N_\sigma(j,N)$, $s=N_\sigma(j,S)$, $w=N_\sigma(j,W)$ and $e=N_\sigma(j,E)$.}};
	\end{tikzpicture}%
	\end{blingbling}
\caption{after step 3: all decorations fixed outside network}
\end{figure}

\begin{figure}[!ht]
\centering
	%%% step 4
	\begin{blingbling}
	\begin{tikzpicture}
	\decotile{f={3},j={1},g={3}}{f={8},j={1},g={8}}{f={p},j={0},g={$w$}}{f={6},j={0},g={$w$}}{$T_4$}{0,0}
	\decotile{f={3},j={2},g={3}}{f={8},j={2},g={8}}{f={p},g={1},j={$j$}}{f={6},g={1},j={$j$}}{$T_4$}{3.5,0}
	\decotile{f={3},j={3},g={3}}{f={8},j={3},g={8}}{f={p},g={2},j={$j$}}{f={6},g={2},j={$j$}}{$T_4$}{7,0}
	\decotile{f={3},j={4},g={3}}{f={8},j={4},g={8}}{f={p},j={$\alpha$},g={$\beta$}}{f={6},j={$\alpha$},g={$\beta$}}{$T_4$}{0,3.5}
	\decotile{f={3},j={5},g={3}}{f={8},j={5},g={8}}{f={p},j={$\alpha$},g={$\beta$}}{f={6},j={$\alpha$},g={$\beta$}}{$T_4$}{3.5,3.5}
	\decotile{f={3},j={6},g={3}}{f={8},j={6},g={8}}{f={p},j={$\alpha$},g={$\beta$}}{f={6},j={$\alpha$},g={$\beta$}}{$T_4$}{7,3.5}
	\decotile{f={3},j={7},g={3}}{f={8},j={7},g={8}}{f={p},j={0},g={$w$}}{f={6},j={0},g={$w$}}{$T_4$}{0,7}
	\decotile{f={3},j={8},g={3}}{f={8},j={8},g={8}}{f={p},g={11},j={$j$}}{f={6},g={11},j={$j$}}{$T_4$}{3.5,7}
	\decotile{f={3},j={9},g={3}}{f={8},j={9},g={8}}{f={p},g={12},j={$j$}}{f={6},g={12},j={$j$}}{$T_4$}{7,7}
	\path (5,-0.2) node[below,black!70] {\begin{minipage}{12cm}
	\emph{where $w=N_\sigma(j,W)$, for any $j$ and any horizontal pair $(\alpha,\beta)$ of parent/neighbor indices}
	\end{minipage}};
	\end{tikzpicture}
	\end{blingbling}
\caption{after step 4: all decorations fixed but for central tiles}
\end{figure}

\begin{figure}[!ht]
\centering
	%%% step 5
	\begin{blingbling}
	\begin{tikzpicture}
	\decotile{f={4},j={0},g={$s$}}{f={9},j={$j$},g={3}}{f={6},j={0},g={$w$}}{f={7},j={$j$},g={1}}{$T_5$}{0,0}
	\decotile{f={4},j={$\alpha^2_j$},g={$\beta^2_j$}}{f={9},j={$\alpha^2_j$},g={$\beta^2_j$}}{f={6},j={$j$},g={1}}{f={7},j={$j$},g={2}}{$T_5$}{3.5,0}
	\decotile{f={4},j={0},g={$s$}}{f={9},j={$j$},g={5}}{f={6},j={$j$},g={2}}{f={7},j={0},g={$e$}}{$T_5$}{7,0}
	\decotile{f={4},j={$j$},g={3}}{f={9},j={$j$},g={8}}{f={6},j={$\alpha^4_j$},g={$\beta^4_j$}}{f={7},j={$\alpha^4_j$},g={$\beta^4_j$}}{$T_5$}{0,3.5}
	\decotile{f={4},j={$j$},g={5}}{f={9},j={$j$},g={10}}{f={6},j={$\alpha^6_j$},g={$\beta^6_j$}}{f={7},j={$\alpha^6_j$},g={$\beta^6_j$}}{$T_5$}{7,3.5}
	\decotile{f={4},j={$j$},g={8}}{f={9},j={0},g={$n$}}{f={6},j={0},g={$w$}}{f={7},j={$j$},g={11}}{$T_5$}{0,7}
	\decotile{f={4},j={$\alpha^8_j$},g={$\beta^8_j$}}{f={9},j={$\alpha^8_j$},g={$\beta^8_j$}}{f={6},j={$j$},g={11}}{f={7},j={$j$},g={12}}{$T_5$}{3.5,7}
	\decotile{f={4},j={$j$},g={10}}{f={9},j={0},g={$n$}}{f={6},j={$j$},g={12}}{f={7},j={0},g={$e$}}{$T_5$}{7,7}
	\path (5,-0.2) node[below,black!70] {\begin{minipage}{12cm}
	\emph{where, for any $j$, $n=N_\sigma(j,N)$, $s=N_\sigma(j,S)$, $w=N_\sigma(j,W)$, $e=N_\sigma(j,E)$, $(\alpha^k_j,\beta^k_j)$ are valid parent/neighbor indices for a tile $T_k$ with parent-index $j$}
	\end{minipage}};
	\end{tikzpicture}%
	\end{blingbling}
\caption{after step 5: all decorations are fixed}
\end{figure}

\begin{thebibliography}{1}

\bibitem{AI}
P.~Arnoux and S.~Ito, \emph{Pisot substitutions and {R}auzy fractals}, Bull.
  Belg. Math. Soc. Simon Stevin \textbf{8} (2001), no.~2, 181--207.

\bibitem{RB}
R.~Berger, \emph{The undecidability of the domino problem}, Ph.D. thesis,
  Harvard University, July 1964.

\bibitem{BF}
V.~Berth{\'e} and Th. Fernique, \emph{Brun expansions of stepped surfaces}, Disc. Math. \textbf{311} (2011), no.~7, 521--543.

\bibitem{TF}
Th.~Fernique, \emph{Local rule substitutions and stepped surfaces}, Theor. Comput. Sci. \textbf{380} (2007), no.~3, 317--329.

\bibitem{GS}
C.~Goodman-Strauss, \emph{Matching rules and substitution tilings}, Ann. of
  Math. (2) \textbf{147} (1998), no.~1, 181--223.

\bibitem{SM}
S.~Mozes, \emph{Tilings, substitution systems and dynamical systems generated
  by them}, J. Analyse Math. \textbf{53} (1989), 139--186.

\bibitem{NO}
N.~Ollinger, \emph{Two-by-two substitution systems and the undecidability of
  the domino problem}, Proceedings of CiE'2008, LNCS, vol. 5028, Springer,
  2008, pp.~476--485.

\bibitem{PF}
N.~Priebe~Frank, \emph{Detecting combinatorial hierarchy in tilings using
  derived voronoi tessellations}, Disc. Comput. Geom. \textbf{29} (2003),
  no.~3, 459--467.

\bibitem{RR}
R.~M. Robinson, \emph{Undecidability and nonperiodicity for tilings of the
  plane}, Inventiones Mathematicae \textbf{12} (1971), 177--209.

\end{thebibliography}
\end{document}